\documentclass{article}
\usepackage{amssymb,amsmath,amsthm,graphicx}
\usepackage[all,color]{xy}
\usepackage{float}

\def\qed{\hfill {\hbox{${\vcenter{\vbox{               
   \hrule height 0.4pt\hbox{\vrule width 0.4pt height 6pt
   \kern5pt\vrule width 0.4pt}\hrule height 0.4pt}}}$}}}

\newtheorem{theorem}{Theorem}

\newtheorem{proposition}[theorem]{Proposition}

\theoremstyle{definition}
\newtheorem{example}{Example}
\newtheorem{definition}{Definition}

\def\qed{\hfill {\hbox{${\vcenter{\vbox{               
   \hrule height 0.4pt\hbox{\vrule width 0.4pt height 6pt
   \kern5pt\vrule width 0.4pt}\hrule height 0.4pt}}}$}}}

\def\tr{\triangleright}

\date{}

\title{\Large \textbf{New Polynomial Decategorifications of the Quandle Coloring Quiver}}

\author{Anusha Kabra\footnote{Email: akabra50@students.claremontmckenna.edu} \and
Sam Nelson\footnote{Email: Sam.Nelson@cmc.edu. Partially supported by Simons Foundation collaboration grant 702597}}

\begin{document}
\maketitle

\begin{abstract}
We introduce two new families of polynomial invariants of oriented classical 
and virtual knots and links defined as decategorifications of the quandle 
coloring quiver. We provide examples to illustrate the computation of the 
invariants, show that they are not determined by the quandle counting 
invariant, and compute tables of invariants values for some small quandles.
\end{abstract}

\parbox{5.5in} {\textsc{Keywords:} quandles, quivers, enhancements of counting invariants, knot invariants

\smallskip

}

\section{\large\textbf{Introduction}}\label{I}

Introduced in \cite{J} (and independently in \cite{M}, see also \cite{B,FR,T}),
\textit{quandles} are algebraic structures whose axioms encode the 
Reidemeister moves of knot theory. Joyce in \cite{J} showed that the 
knot quandle is a complete invariant up to ambient homeomorphism. Since 
then, quandles have been used to define new knot invariants and to study 
relationships between known link invariants. See \cite{EN} for more.

\textit{Quandle coloring quivers,} introduced in \cite{CN} by the second author
and a coauthor, are quiver-valued invariants of oriented classical
and virtual knots and links which categorify the quandle homset invariant
in the sense that quivers are categories and the quandle homset is the vertex
set of the quandle coloring quiver.

Since large quivers can be difficult to compare, a polynomial invariant
of classical and virtual knots and links were defined in \cite{CN}
as decategorifications of the quandle coloring quiver. The
\textit{in-degree polynomial} uses the observation that while the vertices
of a quandle coloring quiver all have the same out-degree, they can have 
different in-degrees. While these polynomials are easier to compare than the
quivers, passing from the quiver to the polynomial loses some information
in the sense that there are examples of pairs of knots with the same
in-degree polynomial with a respect to a fixed choice of quandle and 
endomorphism set which have non-isomorphic quivers.

In this paper we introduce two new infinite families of polynomial 
invariants of classical and virtual knots and link obtained from the 
quandle coloring quiver via decategorification. The first is a two-variable
version of the in-degree polynomial which specializes to the original in-degree 
polynomial but is a stronger invariant in general. The second is a 
single-variable polynomial which uses a different aspect of the structure
of the quiver, namely its set of maximal non-repeating paths.

The paper is organized as follows. In Section \ref{R} we review the basics
of quandle coloring quivers. in Section \ref{ID} we introduce the two-variable
version of the in-degree polynomial and gives examples to show that it is
a proper enhancement of both the in-degree polynomial and the quandle
counting invariant. in Section \ref{MP} we introduce the maximal path
polynomial and provide examples to show that it is not determined by the 
quandle counting invariant. We conclude in Section \ref{Q} with some questions
for future research.

This paper, including all text, diagrams and computational code,
was produced strictly by the authors without the use of generative AI in any 
form.

\section{\large\textbf{Review of Quandle Coloring Quivers}}\label{R}

In this section we review the basics of quandles, quandle colorings and 
quandle coloring quivers. See \cite{CN,EN} for more.

\subsection{Quandles}

A \textit{quandle} is a set $X$ equipped with a binary operation 
$\triangleright: X\times X \to X $ satisfying the following axioms:

\begin{enumerate}
    \item {Idempotency:}
    \[
    \forall x \in X, \quad x \triangleright x = x.
    \]    
    \item {Right Invertibility:} 
    \[
    \forall x, y \in X, \exists\mathrm{\ a \ unique\ }  z \in X  
\mathrm{\ such\  that\ } x = z \triangleright y.
    \]
    \item {Right Self-Distributivity:}
    \[
    \forall x,y,z \in X, \ (x \triangleright y) \triangleright z = (x \triangleright z) \triangleright (y \triangleright z)
    \]
\end{enumerate}

\begin{example} Every oriented classical or virtual link $L$ has a 
\textit{knot quandle} or \textit{fundamental quandle} $\mathcal{Q}(L)$ 
defined from a diagram $D$ via a presentation with generators corresponding 
to arcs and relations corresponding to crossings.

The knot quandle of the trefoil knot $Q(3_1)$ can be computed by representing 
it visually and labeling each arc $x, y, z$ as below. 

The trefoil knot has knot quandle presented by 
\[
\mathcal{Q}(3_1)=\langle x, y, z \mid x = y \triangleright z, \quad 
y = z \triangleright x, \quad z = x \triangleright y \rangle.
\]
\end{example}

It can be visualized as shown:
\[\includegraphics{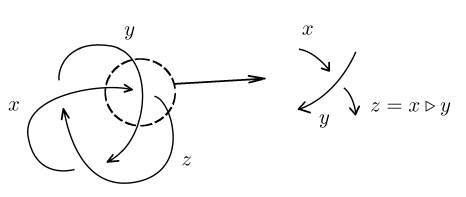}\]

\begin{example} \label{ex:d3}
Finite quandles can be represented by operation tables similar to the example 
shown below for the \textit{dihedral quandle} with $n=3$ where $n$ is the 
number of elements and we define $x\tr y=2y-x$ mod $n$ (writing the class of 
zero mod $n$ as $n$). This quandle will be relevant to the next subsection.

\[
\begin{array}{c|ccc}
\triangleright & 1 & 2 & 3 \\ \hline
1 & 1 & 3 & 2 \\
2 & 3 & 2 & 1 \\
3 & 2 & 1 & 2
\end{array}
\]
\end{example}

\subsection{Quandle Colorings}

A function $f : X \rightarrow Y$ between quandles is called a 
\textit{quandle homomorphism} if for all $x,y \in X$ we have

\[
f(x\triangleright y) = f(x) \triangleright f(y)
\]

A bijective homomorphism is called an \textit{isomorphism}, a self-homomorphism
is called an \textit{endomorphism} and a 
self-isomorphism $f : X \rightarrow X $ is called an \textit{automorphism}.

A quandle homomorphism $f:\mathcal{Q}(K) \to X$ assigns elements of $X$ to 
the generators of $\mathcal{Q}(K)$, coloring each arc of $K$. If arcs are 
labeled as shown, the homomorphism condition  
\[
f(x \triangleright y) = f(x) \triangleright f(y)
\]
must hold, ensuring that $f$ defines a valid coloring. Conversely, any assignment of elements of $X$ to arcs in $K$ that satisfies the crossing relations induces a quandle homomorphism. Thus, quandle colorings provide both a visualization and a computational tool for quandle homomorphisms. 

Let $L$ be an oriented knot or link and $X$ a finite quandle, called the 
\textit{coloring quandle}. The \textit{coloring space} or \textit{homset} is 
the set of quandle homomorphisms 
\[
\mathrm{Hom}(\mathcal{Q}(L), X)=\{f:\mathcal{Q}(L)\to X \ |\ 
f(x \triangleright y) = f(x) \triangleright f(y)\}.
\]
The \textit{quandle counting invariant} is the cardinality of this space, 
$|\text{Hom}(Q(L), X)| $ denoted by $\Phi_X^{\mathbb{Z}}(L)$.

Each homomorphism $\varphi \in \mathrm{Hom}(\mathcal{Q}(L), X)$ corresponds to 
a \textit{coloring} of $L$, assigning elements of $X$ to arcs in a 
link diagram while preserving the quandle operation at crossings. Since each 
arc represents a generator of $\mathcal{Q}(L)$, and $X$ is finite, the number 
of valid colorings is always an integer. Henceforth, we interpret 
homomorphisms in $\text{Hom}(\mathcal{Q}(L), X)$ as colorings of link diagrams.

\begin{example} \label{ex:fox31}
We picture quandle colorings and find the quandle counting invariant 
$|\text{Hom}(Q(L), X)| $ where L is the trefoil and X is the dihedral 
quandle with $n=3$, i.e., the quandle with operation table
\[\begin{array}{r|rrr}
\tr & 1 & 2 & 3 \\ \hline
1 & 1 & 3 & 2 \\
2 & 3 & 2 & 1 \\
3 & 2 & 1 & 3 
\end{array}.\] 
To count the homomorphisms and find the invariant, we need to find the unique 
ways we can validly map the generators $x, y, z$ to elements of $X$. As seen  
earlier, for  the trefoil knot we must satisfy $z = x \triangleright y$, so 
once colors for $x$ and $y$ are chosen, the color for $z$ is determined by the 
operation table of X. 
Below are 9 unique colorings with each arc labeled and each generator 
corresponding to a particular color in $X$.
\[\scalebox{0.9}{\includegraphics{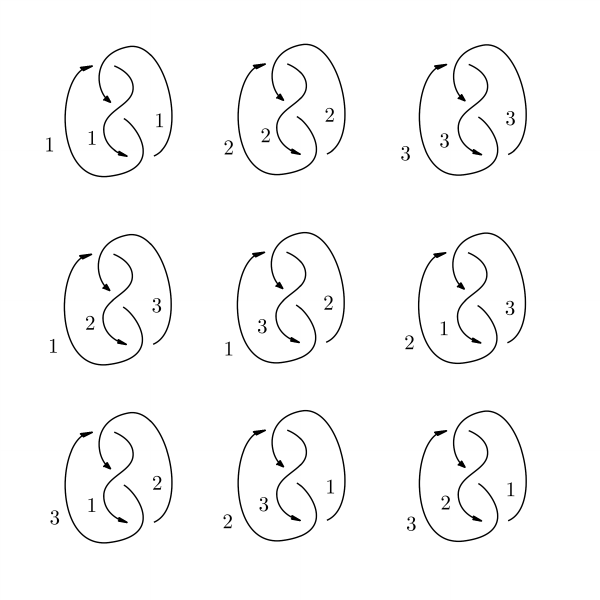}}\]
\end{example}

\subsection{Quandle Coloring Quivers}

A \textit{quandle coloring quiver} is constructed from the set of quandle 
colorings of an oriented knot or link diagram with respect to a finite 
quandle. Each vertex in the quiver represents a unique quandle coloring 
of the knot or link. The edges between vertices are defined by applying 
\textit{endomorphisms} of the coloring quandle, i.e. maps preserving
the quandle operation, to the colorings. This means that 
for each endomorphism $f:X\to X$, there is a directed edge from one 
coloring to another if the endomorphism maps the first coloring to the 
second. More formally:

\begin{definition}
Let $L$ be an oriented classical or virtual knot or link, $X$ a finite
quandle and $S\subseteq \mathrm{Hom}(X,X)$. The \textit{quandle 
coloring quiver} of $X$ with respect to $(X,S)$ is the directed graph 
$\mathcal{QCQ}_{X,S}(L)$
with a  vertex for each $f\in\mathrm{Hom}(\mathcal{Q}(L),X)$ and a directed 
edge from $f$ to $g$ whenever $f=\sigma g$ for some $\sigma\in S$.
\end{definition}

\begin{example}
To compute the quandle coloring quiver for a given link, we can find all of the
$X$-colorings of a diagram of our link and then draw arrows between the 
colorings that are related by endomorphisms. Then for example the maps
$\sigma_1=[1,1,1]$ and $\sigma_2(x)=[2,1,3]$ (where the notation is
$\sigma=[\sigma(1),\dots,\sigma(n)]$) are both endomorphisms of the dihedral 
quandle with $n=3$ from Example \ref{ex:d3}; applying these endomorphisms
to the quandle colorings of the trefoil in Example \ref{ex:fox31}, we obtain 
quandle coloring quiver $\mathcal{QCQ}_{X,S}(3_1)$ where 
$S=\{\sigma_1,\sigma_2\}$ as shown:
\[\includegraphics{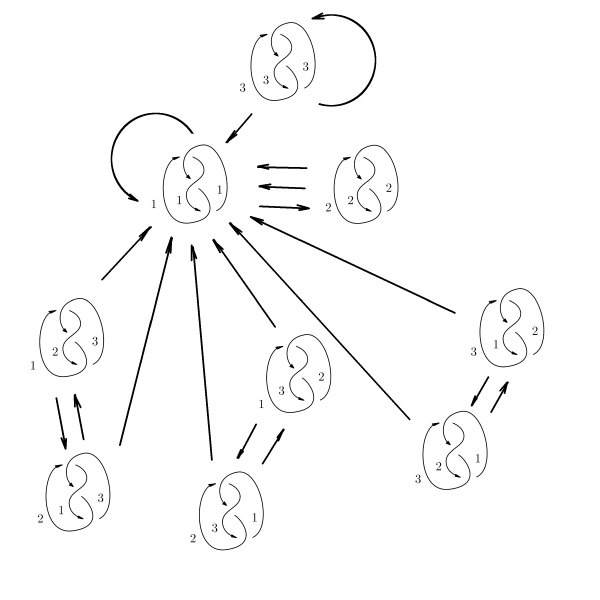}\]
\end{example}

See \cite{CN} for more.

\section{\large\textbf{2-Variable In-Degree Polynomial}}\label{ID}

We begin this section with a definition.

\begin{definition}
Let $X$ be a finite quandle, $S\subset \mathrm{Hom}(X,X)$ a set of 
endomorphisms of $X$, and $L$ an oriented classical or virtual knot
or link represented by a diagram $D$. We define the \textit{two-variable
in-degree polynomial} of $L$ with respect to $(X,S)$, denoted
$\Phi_{X,S}^{2\mathrm{deg}^+}(L)$, to be 
\[\Phi_{X,S}^{2\mathrm{deg}^+}(L)
=\sum_{e\in E(\mathcal{Q}_{X,S}(L))} s^{\mathrm{deg}_+(S(e))}t^{\mathrm{deg}_+(T(e))}\]
where $S(e)$ and $T(e)$ are the vertices at the source and target of the
edge $e$ respectively and $E(\mathcal{Q}_{X,S}(L))$ is the set of edges 
in the quiver.
\end{definition}

We then have:
\begin{proposition}
The two-variable in-degree polynomial is an invariant of oriented 
classical and virtual knots and links.
\end{proposition}

\begin{proof}
Changing the diagram $D$ by Reidemeister moves does not change the quandle 
coloring quiver; it follows immediately that the set of edges and the 
in-degrees of the vertices at their endpoints are also invariant under 
Reidemeister moves.
\end{proof}

\begin{example} \label{ex1}
Let $X$ be the quandle with operation table
\[\begin{array}{r|rrr}
\tr & 1 & 2 & 3 \\ \hline
1 & 1 & 1 & 2 \\
2 & 2 & 2 & 1 \\
3 & 3 & 3 & 3
\end{array}
\]
and consider the Hopf link $L$. There are five $X$-colorings of $L$; 
selecting the endomorphisms $S=\{[1,1,2],[2,2,1]\}$ we obtain quandle 
coloring quiver
\[\includegraphics{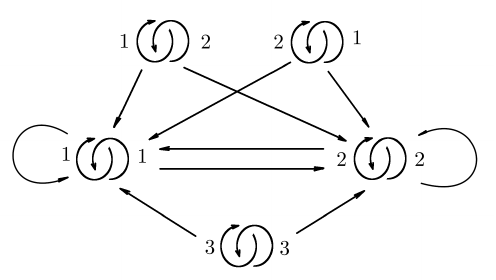}.\]
Replacing each diagram with a vertex labeled with its in-degree, we have
\[\includegraphics{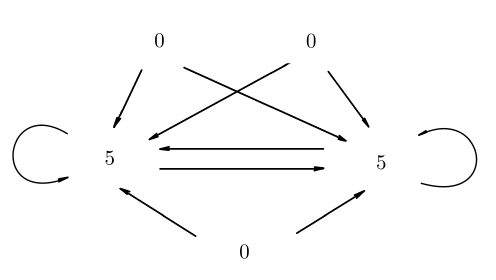}.\]
Then summing over the edges we have two-variable in-degree polynomial
$\Phi_{X,S}^{2\mathrm{deg}^+}(L)=6s^0t^5+4s^5t^5$.
\end{example}

The 2-variable in-degree polynomial is an \textit{enhancement} of the 
quandle counting invariant, meaning that the quandle counting invariant 
can be recovered from the 2-variable in-degree polynomial. This follows 
from the fact that the original in-degree polynomial in an enhancement of
the quandle counting invariant and the following proposition.


\begin{proposition}
The two-variable in-degree polynomial $\Phi_{X,S}^{2\mathrm{deg}^+}(L)$ 
satisfies
\[\left.\Phi_{X,S}^{2\mathrm{deg}^+}(L)\right|_{s=1} 
= |S|\Phi_{X,S}^{\mathrm{deg}^+}(L);\]
that is, the in-degree polynomial $\Phi_{X,S}^{\mathrm{deg}^+}(L)$ can be obtained
from the two-variable in-degree polynomial by evaluating at $s=1$ and dividing 
by $|S|$.
\end{proposition}

The next example shows that the two-variable in-degree polynomial is a proper
enhancement both of the quandle counting invariant and the original
in-degree polynomial.

\begin{example}
Let $X$ be the quandle with operation table
\[\begin{array}{r|rrrr}
\tr & 1 & 2 & 3 & 4 \\ \hline
1 & 1 & 1 & 4 & 3 \\
2 & 2 & 2 & 2 & 2 \\
3 & 4 & 3 & 3 & 1 \\
4 & 3 & 4 & 1 & 4
\end{array}\]
and let $S=\{[1,2,4,3],[2,1,2,2]\}$. Then our \texttt{python} computations
show that the links $L7a3$ and $L7a5$ both have 16 $X$-colorings and both 
have in-degree polynomial value 
\[\Phi_{X,S}^{\mathrm{deg}^+}(L7a3)
=12t+2t^2+8t^4+10t^{10}
=\Phi_{X,S}^{\mathrm{deg}^+}(L7a5)\]
while the two-variable polynomial values
\[\Phi_{X,S}^{2\mathrm{deg}^+}(L7a3)
=s^{10}t^{10}+s^{10}t^4+s^4t^{10}+2s^4t^4+s^4t^2+s^2t^4+s^2t^2+8st^{10}+4st^4+12st
\]
and
\[\Phi_{X,S}^{2\mathrm{deg}^+}(L7a5)
=s^{10}t^{10}+s^{10}t^2+4s^4t^4+s^2t^{10}+s^2t^2+8st^{10}+4st^4+12st
\]
distinguish the links.
\end{example}

\begin{example} \label{ex:table1}
Using our \texttt{python} code, we computed the 2-variable in-degree 
polynomial for the prime classical links with up to seven crossings as 
found in \cite{KA} for the quandle $X$ with operation table 
\[\begin{array}{r|rrr}
\tr & 1 & 2 & 3 \\ \hline
1 & 1 & 1 & 2 \\
2 & 2 & 2 & 1 \\
3 & 3 & 3 & 3
\end{array}
\] 
arising from the full quiver, i.e., using the entire
endomorphism set. The results are in the table.
\[\scalebox{0.95}{$
\begin{array}{r|l}
L & \Phi_{X,S}^{2\mathrm{deg}^+}(L) \\ \hline
L2a1 & 12s^{12}t^{12} + 2s^{12}t^7 + 4s^7t^{12} + 3s^7t^7 + 8s^2t^{12} + 2s^2t^7 + 4s^2t^2\\
L4a1 & 12s^{16}t^{16} + 2s^{16}t^{11} + 4s^{11}t^{16} + 3s^{11}t^{11} + 8s^6t^{16} + 2s^6t^{11} + 4s^6t^6 + 8s^2t^{16} + 4s^2t^{11} + 8s^2t^6 + 8s^2t^2\\
L5a1 & 12s^{16}t^{16} + 2s^{16}t^{11} + 4s^{11}t^{16} + 3s^{11}t^{11} + 8s^6t^{16} + 2s^6t^{11} + 4s^6t^6 + 8s^2t^{16} + 4s^2t^{11} + 8s^2t^6 + 8s^2t^2\\
L6a1 & 12s^{16}t^{16} + 2s^{16}t^{11} + 4s^{11}t^{16} + 3s^{11}t^{11} + 8s^6t^{16} + 2s^6t^{11} + 4s^6t^6 + 8s^2t^{16} + 4s^2t^{11} + 8s^2t^6 + 8s^2t^2\\
L6a2 & 12s^{12}t^{12} + 2s^{12}t^7 + 4s^7t^{12} + 3s^7t^7 + 8s^2t^{12} + 2s^2t^7 + 4s^2t^2\\
L6a3 & 12s^{12}t^{12} + 2s^{12}t^7 + 4s^7t^{12} + 3s^7t^7 + 8s^2t^{12} + 2s^2t^7 + 4s^2t^2\\
L6a4 & 12s^{38}t^{38} + 2s^{38}t^{29} + 4s^{29}t^{38} + 3s^{29}t^{29} + 24s^8t^{38} + 6s^8t^{29} + 12s^8t^8 + 36s^2t^{38} + 18s^2t^{29} + 36s^2t^8 + 36s^2t^2\\
L6a5 & 12s^{26}t^{26} + 2s^{26}t^{17} + 4s^{17}t^{26} + 3s^{17}t^{17} + 24s^4t^{26} + 6s^4t^{17} + 12s^4t^4 + 12s^2t^{26} + 6s^2t^{17} + 12s^2t^4 + 12s^2t^2\\
L6n1 & 12s^{26}t^{26} + 2s^{26}t^{17} + 4s^{17}t^{26} + 3s^{17}t^{17} + 24s^4t^{26} + 6s^4t^{17} + 12s^4t^4 + 12s^2t^{26} + 6s^2t^{17} + 12s^2t^4 + 12s^2t^2\\
L7a1 & 12s^{16}t^{16} + 2s^{16}t^{11} + 4s^{11}t^{16} + 3s^{11}t^{11} + 8s^6t^{16} + 2s^6t^{11} + 4s^6t^6 + 8s^2t^{16} + 4s^2t^{11} + 8s^2t^6 + 8s^2t^2\\
L7a2 & 12s^{16}t^{16} + 2s^{16}t^{11} + 4s^{11}t^{16} + 3s^{11}t^{11} + 8s^6t^{16} + 2s^6t^{11} + 4s^6t^6 + 8s^2t^{16} + 4s^2t^{11} + 8s^2t^6 + 8s^2t^2\\
L7a3 & 12s^{16}t^{16} + 2s^{16}t^{11} + 4s^{11}t^{16} + 3s^{11}t^{11} + 8s^6t^{16} + 2s^6t^{11} + 4s^6t^6 + 8s^2t^{16} + 4s^2t^{11} + 8s^2t^6 + 8s^2t^2\\
L7a4 & 12s^{16}t^{16} + 2s^{16}t^{11} + 4s^{11}t^{16} + 3s^{11}t^{11} + 8s^6t^{16} + 2s^6t^{11} + 4s^6t^6 + 8s^2t^{16} + 4s^2t^{11} + 8s^2t^6 + 8s^2t^2\\
L7a5 & 12s^{12}t^{12} + 2s^{12}t^7 + 4s^7t^{12} + 3s^7t^7 + 8s^2t^{12} + 2s^2t^7 + 4s^2t^2\\
L7a6 & 12s^{12}t^{12} + 2s^{12}t^7 + 4s^7t^{12} + 3s^7t^7 + 8s^2t^{12} + 2s^2t^7 + 4s^2t^2\\
L7a7 & 12s^{26}t^{26} + 2s^{26}t^{17} + 4s^{17}t^{26} + 3s^{17}t^{17} + 24s^4t^{26} + 6s^4t^{17} + 12s^4t^4 + 12s^2t^{26} + 6s^2t^{17} + 12s^2t^4 + 12s^2t^2\\
L7n1 & 12s^{16}t^{16} + 2s^{16}t^{11} + 4s^{11}t^{16} + 3s^{11}t^{11} + 8s^6t^{16} + 2s^6t^{11} + 4s^6t^6 + 8s^2t^{16} + 4s^2t^{11} + 8s^2t^6 + 8s^2t^2\\
L7n2 & 12s^{16}t^{16} + 2s^{16}t^{11} + 4s^{11}t^{16} + 3s^{11}t^{11} + 8s^6t^{16} + 2s^6t^{11} + 4s^6t^6 + 8s^2t^{16} + 4s^2t^{11} + 8s^2t^6 + 8s^2t^2
\end{array}$}
\]
\end{example}

\section{\large\textbf{Maximal Path Polynomial}}\label{MP}

\begin{definition}
Let $X$ be a finite quandle, $S\subset \mathrm{Hom}(X,X)$ a set of 
endomorphisms of $X$, and $L$ an oriented classical or virtual knot
or link represented by a diagram $D$. We define the \textit{maximal path
polynomial} of $L$ with respect to $X,S$ to be the sum over the set of 
maximal non-repeating paths $p$ in $\mathcal{Q}_{X,S}(L)$ of $q^{|p|}$
where $|p|$ is the length of $p$, i.e.,
\[\Phi_{X,S}^{MP}(L)=\sum_{p\in MP} x+q^{|p|}.\]
\end{definition}

\begin{example}
Consider the quandle coloring quiver $\mathcal{Q}_{X,S}(L)$ 
from Example \ref{ex1}. 
\[\scalebox{0.9}{\includegraphics{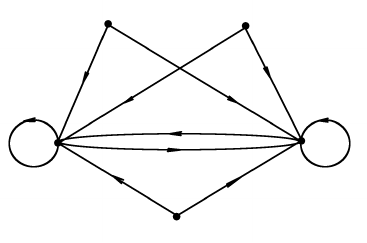}}\]
There is a central cycle of length 4 and a total of six edges leading into 
the cycle; from each such edge, there are two possible ways to go around 
the central cycle, namely starting with the loop first or coming back to 
it at the end. Thus, we have a total of 12 maximal non-repeating paths,
each of which has length 5. We illustrate two such paths by numbering edges 
in order:
\[\scalebox{0.9}{\includegraphics{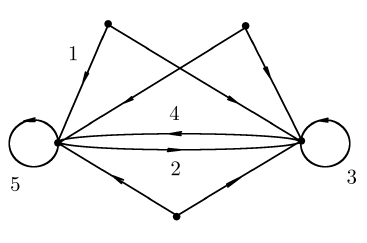}\quad \includegraphics{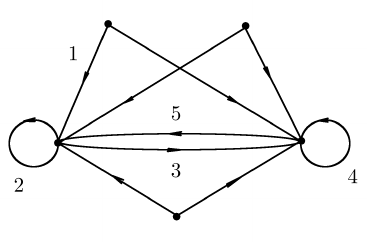}}\]
Then the maximal path polynomial for the Hopf link with respect to the quandle 
and endomorphism set from Example \ref{ex1} is $\Phi_{X,S}^{MP}(L2a1)=12q^5$.
\end{example}

The next example shows that the maximal path polynomial is not determined by 
the quandle counting invariant.

\begin{example}
Let $X$ be the quandle with operation table
\[\begin{array}{r|rrrrr}
\tr & 1 & 2 & 3 & 4 & 5 \\ \hline
1 & 1 & 3 & 2 & 2 & 3 \\
2 & 3 & 2 & 1 & 3 & 1 \\
3 & 2 & 1 & 3 & 1 & 2 \\
4 & 5 & 5 & 5 & 4 & 4 \\
5 & 4 & 4 & 4 & 5 & 5
\end{array}\]
and let $S=\{[1,3,2,5,4],[5,5,5,4,4]\}$. Then our \texttt{python}
computations show that the links $L6a3$ and $L7a2$ both have 13
$X$-colorings but are distinguished by the associated maximal path polynomial
with
\[\Phi_{X,S}^{MP}(L6a3)=12q^6+32q^7\]
while
\[\Phi_{X,S}^{MP}(L7a2)=8q^6+32q^7+24q^8+24q^9.\]
\end{example}

\begin{example} \label{ex:table2}
Using 
our \texttt{python} code we computed the maximal path  
polynomial for the prime classical links with up to seven crossings as 
found in \cite{KA} for the quandle $X$ with operation table 
\[\begin{array}{r|rrrr}
\tr & 1 & 2 & 3 \\ \hline
1 & 1 & 1 & 2 \\
2 & 2 & 2 & 1 \\
3 & 3 & 3 & 3
\end{array}
\] 
arising from the quiver with endomorphism set $S=\{[2,4,3,1],[2,2,4,2]\}$. 
The results are in the tables.
\[
\begin{array}{r|l}
L & \Phi_{X,S}^{MP}(L) \\ \hline
L2a1 & 24q^8 + 20q^7 + 4q^5 \\
L4a1 & 24q^8 + 20q^7 + 4q^5 \\
L5a1 & 24q^{12} + 36q^{11} + 36q^{10} + 36q^9 + 28q^8 + 16q^7 + 4q^5 \\
L6a1 & 24q^8 + 20q^7 + 4q^5 \\
L6a2 & 24q^{12} + 36q^{11} + 36q^{10} + 36q^9 + 28q^8 + 16q^7 + 4q^5 \\
L6a3 & 24q^{12} + 36q^{11} + 36q^{10} + 36q^9 + 28q^8 + 16q^7 + 4q^5 \\
L6a4 & 144q^{12} + 216q^{11} + 216q^{10} + 216q^9 + 144q^8 + 44q^7 + 4q^5 \\
L6a5 & 96q^8 + 56q^7 + 4q^5 \\
L6n1 & 96q^8 + 56q^7 + 4q^5 \\
\end{array} \]
\[
\begin{array}{r|l}
L & \Phi_{X,S}^{MP}(L) \\ \hline
L7a1 & 24q^{12} + 36q^{11} + 36q^{10} + 36q^9 + 28q^8 + 16q^7 + 4q^5 \\
L7a2 & 24q^8 + 20q^7 + 4q^5 \\
L7a3 & 24q^{12} + 36q^{11} + 36q^{10} + 36q^9 + 28q^8 + 16q^7 + 4q^5 \\
L7a4 & 24q^{12} + 36q^{11} + 36q^{10} + 36q^9 + 28q^8 + 16q^7 + 4q^5 \\
L7a5 & 24q^8 + 20q^7 + 4q^5 \\
L7a6 & 24q^8 + 20q^7 + 4q^5 \\
L7a7 & 24q^{12} + 36q^{11} + 36q^{10} + 36q^9 + 100q^8 + 52q^7 + 4q^5 \\
L7n1 & 24q^8 + 20q^7 + 4q^5 \\
L7n2 & 24q^{12} + 36q^{11} + 36q^{10} + 36q^9 + 28q^8 + 16q^7 + 4q^5.
\end{array}
\]

\end{example}

\section{\large\textbf{Questions}}\label{Q}

The examples given in this paper are toy examples in the sense that they use
only small quandles, small sets of endomorphisms and small crossing-number
knots and links for the sake of quick computability. Fast algorithms
for working with larger quandles, larger quivers and larger knots and links
are of great interest.

Of course the primary question is ``What other decategorifications and 
further enhancements can be derived from quandle coloring quivers?'' We 
anticipate many more examples of these in future work.

The in-degree polynomials are enhancements of the quandle counting invariant
in the sense that they determine the quandle counting invariant. Is the same
true of the maximal path polynomial? It seems like it shouldn't be true,
but it is not clear how to reconstruct a quiver from a maximal path 
polynomial.

\bibliography{ak-sn}{}
\bibliographystyle{abbrv}

\bigskip

\medskip

\noindent
\textsc{Department of Mathematical Sciences \\
Claremont McKenna College \\
850 Columbia Ave. \\
Claremont, CA 91711}

\end{document}